\documentclass[12pt,reqno]{amsart}
  \usepackage[all,2cell,arrow]{xy}
  \usepackage{amssymb}
    \usepackage{tikz}
 \numberwithin{equation}{section}
 \usepackage{color}
\definecolor{darkgreen}{rgb}{0,0.45,0} 
\usepackage[pagebackref,colorlinks,citecolor=darkgreen,linkcolor=darkgreen]{hyperref}
\usepackage{enumerate,xspace}

\UseAllTwocells
\CompileMatrices

\newcommand{\cc}{\ensuremath{\mathcal C}\xspace}

\newcommand{\Vect}{\textsf{Vect}\xspace}

\newcommand{\Hopf}{{\sf Hopf}\xspace}
\newcommand{\HKer}{{\sf HKer}\xspace}

\renewcommand{\phi}{\varphi}

\def\1c#1{\stackrel{#1}{\to}}

\newcommand{\ot}{\otimes}

  \newtheorem{proposition}{Proposition}[section]
  \newtheorem{lemma}[proposition]{Lemma}
  
  \newtheorem{corollary}[proposition]{Corollary}
  \newtheorem{theorem}[proposition]{Theorem}

  \theoremstyle{definition}
  \newtheorem{definition}[proposition]{Definition}

  \theoremstyle{remark}
  \newtheorem{remark}[proposition]{Remark}

  \newcounter{c}
  \renewcommand{\[}{\setcounter{c}{1}$$}
  \newcommand{\etyk}[1]{\vspace{-7.4mm}$$\begin{equation}\Label{#1}
  \addtocounter{c}{1}}
  \renewcommand{\]}{\ifnum \value{c}=1 $$\else \end{equation}\fi}
\begin{document}

 \title{A semi-abelian extension of a theorem by Takeuchi}

\author{Marino Gran, Florence Sterck}
\address{Institut de Recherche en Math\'ematique et Physique, Universit\'e catholique de Louvain, Chemin du Cyclotron 2, 1348 Louvain-la-Neuve, Belgium}
\email{marino.gran@uclouvain.be, florence.sterck@uclouvain.be}

\author{Joost Vercruysse}
\address{D\'epartement de Math\'ematique, Universit\'e Libre de Bruxelles, Belgium}
\email{jvercruy@ulb.ac.be}

\thanks{ }

\date{\today}

\begin{abstract}
We prove that the category of cocommutative Hopf algebras over a field is a semi-abelian category. This result extends a previous special case of it, based on the Milnor-Moore theorem, where the field was assumed to have zero characteristic. Takeuchi's theorem asserting that the category of commutative and cocommutative Hopf algebras over a field is abelian immediately follows from this new observation. We also prove that the category of cocommutative Hopf algebras over a field is action representable.
  We make some new observations concerning the categorical commutator of normal Hopf subalgebras, and this leads to the proof that two definitions of crossed modules of cocommutative Hopf algebras are equivalent in this context.
\end{abstract}  
\maketitle


\section{Introduction}\label{sect:torsion-free}
The category $\Hopf_{K, coc}$ of cocommutative Hopf algebras over a field $K$ is known to share many exactness properties with the categories of groups and of Lie algebras. It was already noted by Yanagihara in \cite{Yanagihara,Yanagihara2} that some classical isomorphism theorems from group theory have their natural counterpart for cocommutative Hopf algebras. This work was based on a theorem by Newman \cite{New}, establishing a useful correspondence between left ideals that are also coideals and Hopf subalgebras of a given cocommutative Hopf algebra. 

More recently it was observed \cite{GKV} that, when the base field $K$ has characteristic zero, the category $\Hopf_{K, coc}$ is semi-abelian (in the sense of \cite{JMT}). This fact implies that many of the exactness properties of the category of cocommutative Hopf algebras follow directly from the axioms of semi-abelian category, and this has opened the way to explore some new connections between categorical algebra and Hopf algebra theory \cite{GKV2, VW, GVdL}. As a matter of fact, the proof given in \cite{GKV} depended on the Milnor-Moore theorem 
providing a canonical decomposition of any cocommutative Hopf algebra as a semi-direct product (also called smash product in the literature) of a group Hopf algebra acting on a primitively generated Hopf algebra \cite{MM}. This theorem, however, only holds when the characteristic of the field $K$ is zero, so that the fact that $\Hopf_{K, coc}$ is semi-abelian could only be proved under this additional assumption.

 A first goal of this paper is to show that the category $\Hopf_{K, coc}$ of cocommutative Hopf algebras over \emph{any} field $K$ is semi-abelian, and also action representable in the sense of \cite{BJK}. The fact that $\Hopf_{K, coc}$ is semi-abelian can be seen as a non-commutative version of Takeuchi's theorem asserting that the category of commutative and cocommutative Hopf algebras over a field $K$ is abelian \cite{Takeuchi} (that extends its finite dimensional version due to Grothendieck \cite{Sweedler}). Indeed, Takeuchi's theorem easily follows from the fact that $\Hopf_{K, coc}$ is semi-abelian, since the category of abelian objects in $\Hopf_{K, coc}$ is then abelian, and it is isomorphic to the category of commutative and cocommutative Hopf algebras over $K$.
We then prove that the category $\Hopf_{K, coc}$ is also an action representable category \cite{BJK}, and provide an algebraic description of the categorical commutator in the sense of Huq of two normal Hopf subalgebras in $\Hopf_{K, coc}$ is then given. Finally, we use the fact that $\Hopf_{K, coc}$ is semi-abelian to prove that the notion of crossed module of cocommutative Hopf algebras (in the sense of \cite{Vilaboa, Majid}) becomes a special case of the categorical notion of internal crossed module (in the sense of \cite{Jan}) in the category of cocommutative Hopf algebras. One can deduce from these results that the category of crossed modules of cocommutative Hopf algebras is also semi-abelian.
 
 \section{ $\Hopf_{K, coc}$ is a semi-abelian category}\label{preliminaries}
 A morphism $p$ in a category $\cc$ is called a {\em regular epimorphism} if it is the coequalizer of two morphisms in $\cc$. A finitely complete category \cc is \emph{regular} if any arrow $f \colon A \rightarrow B$ factors as a regular epimorphism $p \colon A \rightarrow I$ followed by a monomorphism $i \colon I \rightarrow B$ and if, moreover, these factorizations are pullback-stable. 
A {\em relation} on an object $A$ in $\cc$ is
a triple $(R, r_1,r_2)$, where $R$ is an object of $\cc$ and $r_1,r_2:R\to A$ is a pair of jointly monic morphisms in $\cc$. A relation $(R,r_1,r_2)$ on $A$ is called:
 \begin{itemize}
 \item \emph{reflexive} if there is a (unique) morphism $\delta \colon A \rightarrow R$ such that $r_1 \cdot \delta  = 1_A = r_2 \cdot \delta  $; 
 \item 
 \emph{symmetric} if there is a (unique) morphism $\sigma \colon R \rightarrow R$ such that $r_1 \cdot \sigma = r_2$ $r_2 \cdot \sigma = r_1$;
\item  \emph{transitive} if there is a (unique) morphism $\tau \colon R \times_ A R \rightarrow R$ such that $r_1 \cdot \tau = r_1 \cdot \pi_1$ and $r_2 \cdot \tau = r_2 \cdot \pi_2$, where $(R \times_ A R, \pi_1, \pi_2)$ is the pullback of $r_1$ and $r_2$.\\
 \end{itemize}
 An {\em equivalence relation} in $\cc$ is a relation $R$ on $A$ that is reflexive, symmetric and transitive. A regular category \cc is \emph{exact} (in the sense of Barr \cite{Ba}) if any equivalence relation $R$ in \cc is \emph{effective}, i.e. it is the kernel pair of some morphism in \cc. This property means that for any equivalence relation there is a morphism $f \colon A \rightarrow B$ in \cc such that the following square is a pullback:
 $$
 \xymatrix{R \ar[r]^{r_2} \ar[d]_{r_1} & A \ar[d]^f \\
 A \ar[r]_f & B.
 }
 $$
 
 A \emph{semi-abelian category} is an exact category \cc, that is also pointed (i.e. it has a zero object), finitely cocomplete and \emph{protomodular} in the sense of \cite{Bourn}. In the presence of a zero object, the protomodularity can be expressed by simply asking that the \emph{Split Short Five Lemma} holds in \cc.

Among the examples of semi-abelian categories there are the categories of groups, Lie algebras, (associative) rings, loops, crossed modules, compact groups, Heyting semilattices, ${\mathbb C}^*$-algebras, and the dual of the category of pointed sets.
As explained in \cite{JMT}, semi-abelian categories are suitable to define and study (co)homology of non-abelian structures \cite{EGV}, and to define and develop a categorical approach to commutator and radical theories. In a semi-abelian category there is also a natural notion of semi-direct product, internal action \cite{BJK} and of crossed module \cite{Jan}. We refer to \cite{BB} for more details about the basic notions and properties of semi-abelian categories. 

The following reformulation of the notion of a regular category will be useful:
\begin{lemma}\label{regularity}
Let \cc be a finitely complete category. Then \cc is a regular category if and only if  
\begin{enumerate}
\item any arrow in \cc factors as a regular epimorphism followed by a monomorphism; 
\item given any regular epimorphism $f \colon A \rightarrow B$ and any object $E$, the induced arrow $1_E \times f \colon E \times A \rightarrow E \times B$ is a regular epimorphism;
\item  regular epimorphisms are stable under pullbacks along split mono-morphisms.
\end{enumerate}
\end{lemma}
\begin{proof}
It is well known that the properties $(1)$, $(2)$ and $(3)$ hold in any regular category.

Conversely, if these properties hold, we need to prove that regular epimorphisms are pullback stable. Given any pullback
\begin{equation}\label{pullback}
\xymatrix{E \times_B  A \ar[r]^-{p_2} \ar[d]_{p_1} & A \ar[d]^f \\
E \ar[r]_p & B
}\end{equation}
where $f$ is a regular epimorphism, 
consider the diagram
$$\xymatrix{E \times_B  A \ar[r]^-{e} \ar[d]_{p_1} & E \times A \ar[d]^{1_E \times f} \\
E \ar[r]_-{(1_E,p)} & E \times B
}$$
where $e \colon E \times_B A \rightarrow E \times A$ is the equalizer of $p\cdot \pi_1$ and $f \cdot \pi_2$,and the arrows $\pi_1 \colon E \times A \rightarrow E$ and $\pi_2 \colon  E \times A \rightarrow A$ are the product projections. This diagram is a pullback,
and it then follows that $p_1$ is a regular epimorphism, since it is
the pullback of the regular epimorphism $1_E \times f$ along the split monomorphism $(1_E, p)$.
\end{proof}

 We are now going to study the regularity and the exactness of the category $\Hopf_{K,coc}$ of cocommutative Hopf algebras over an arbitrary fixed field $K$.
 Recall that $K$-coalgebra is a coassociative and counital coalgebra over $K$, that is a vector space $C$ endowed with linear maps $\Delta:C\to C\ot C$ and $\epsilon:C\to K$ satisfying $(id\ot \Delta)\cdot \Delta=(\Delta\ot id)\cdot\Delta$ (coassociativity) and $(id\ot \epsilon)\cdot \Delta=id=(\epsilon\ot id)\cdot\Delta$ (counitality). We use the classical Sweedler notation for calculations with the comultiplication,
and we shall write $\Delta(c)=c_1\ot c_2$ for any $c\in C$ (with the usual summation convention, where $c_1\ot c_2$ stands for $\sum c_1\ot c_2$). 

Coassociativity and counitality can then be expressed by the formulas
\begin{eqnarray*}
&c_1\ot c_{2,1}\ot c_{2,2}=c_{1,1}\ot c_{1,2}\ot c_2=c_1\ot c_2\ot c_3.\\
&c_1\epsilon(c_2)=c=\epsilon(c_1)c_2
\end{eqnarray*}
A two-sided coideal $I$ in a coalgebra $C$ is a $K$-linear subspace $I\subset C$ such that $\Delta(I)\subset I\ot C + C\ot I$ and $\epsilon(I)=0$. For any two-sided coideal $I$, the linear quotient $C/I$ is a coalgebra and the canonical projection $C\to C/I$ is a coalgebra morphism.

Recall that a $K$-bialgebra $(A,M,u,\Delta,\epsilon)$ is an algebra $A$ with multiplication $M:A \ot A\to A$ and unit $u:K\to A$ that is at the same time a coalgebra with comultiplication $\Delta:A\to A\ot A$ and counit $\epsilon: A\to K$ such that $M$ and $u$ are coalgebra morphisms or, equivalently, $\Delta$ and $\epsilon$ are algebra morphisms, which can be expressed in Sweedler notation as
\begin{eqnarray*}
(ab)_1\ot (ab)_2 = a_1b_1\ot a_2b_2 && 1_1\ot 1_2=1\ot 1\\
\epsilon(ab)=\epsilon(a)\epsilon(b) && \epsilon(1)=1.
\end{eqnarray*}
A Hopf $K$-algebra is a sextuple $(A,M,u,\Delta,\epsilon,S)$ where $(A,M,u,\Delta,\epsilon)$ is a bialgebra and $S:A\to A$ is a linear map, called the {\em antipode}, making the following diagrams commute
\[\xymatrix{
& A\otimes A \ar@<0.5ex>[rr]^-{S \otimes id} \ar@<-0.5ex>[rr]_-{id \otimes S} & &  A\otimes A \ar[dr]^-{M}  \\
A \ar[rr]_-{\epsilon} \ar[ur]^-{\Delta} & &K \ar[rr]_-{u}& & A.
   }\]
In the Sweedler notation, the commutativity of these diagrams can be written as
$$a_1S(a_2)=\epsilon(a)1_A=S(a_1)a_2,$$
for any $a\in A$.

A Hopf algebra $(A,M,u,\Delta,\epsilon,S)$ is \textit{cocommutative} if its underlying coalgebra is cocommutative, meaning that the comultiplication map $\Delta$ satisfies $\sigma \cdot \Delta = \Delta$, where $\sigma\colon A \otimes A \longrightarrow A \otimes A$ is the switch map $\sigma(a \otimes b)=(b \otimes a)$, for any $a\otimes b\in A \otimes A$. In Sweeder notation: $a_1\ot a_2=a_2\ot a_1$.

A morphism of Hopf algebras is a linear map that is both an algebra and a coalgebra morphism (the antipode is then automatically preserved). 
$\mathbf{Hopf}_{K,coc}$ is the category whose objects are cocommutative Hopf $K$-algebras and whose morphisms are morphisms of Hopf $K$-algebras.

 The category $\mathbf{Hopf}_{K,coc}$ is complete and cocomplete \cite{Sweedler, Porst}, and pointed, with zero object the base field $K$. $\mathbf{Hopf}_{K,coc}$ can be seen as the category $\mathsf{Grp}(\mathsf{Coalg}_{K, coc})$ of internal groups in the category $\mathsf{Coalg}_{K, coc}$
 of cocommutative coalgebras, since binary products in $\mathsf{Coalg}_{K, coc}$ are tensor products. Moreover, the equalizer of two coalgebra morphisms $f,g : A \rightarrow B$ in $\mathsf{Coalg}_{K, coc}$ always exists (see 2.4.3 in \cite{GP}, for instance). Indeed, the category $\mathsf{Coalg}_{K}$ of \emph{all} $K$-coalgebras has equalizers (since it is locally presentable, see Theorem $9$ in \cite{Porst2}), and any subcoalgebra of $A$ is cocommutative, since $A$ is so. Accordingly, the category $\mathsf{Coalg}_{K, coc}$ has binary products and equalizers, and is then finitely complete. From this remark it follows that $\mathbf{Hopf}_{K,coc}$ is a protomodular category, since so is any category of internal groups in a finitely complete category (see Example $5$ on p. 57 in \cite{Bourn}).
 
 Also in the category $\mathbf{Hopf}_{K,coc}$ the categorical product of two cocommutative Hopf algebras $A$ and $B$ is given by the tensor product
$(A \otimes B, p_1,p_2)$ where the projections $p_1 \colon A \otimes B \rightarrow A$ and $p_2 \colon A \otimes B \rightarrow B$ are defined by $p_1 (a \otimes b) = a \epsilon (b)$ and $p_2 (a \otimes b) = \epsilon (a) b$ (for any $a \otimes b \in A \otimes B$).
More generally,  the ``object part'' $A \times_B  C$ of the pullback 
\begin{equation*}
\xymatrix{A \times_B  C \ar[r]^-{p_2} \ar[d]_{p_1} & C \ar[d]^g \\
A \ar[r]_f & B
}\end{equation*}
of two morphisms $f:A \rightarrow B$ and $g:C \rightarrow B$ in $\Hopf_{K, coc}$ is given by
$$A \times_B C = \{a \otimes c \in A \otimes C \, \mid \,  a_1 \otimes f(a_2) \otimes c = a \otimes g(c_1) \otimes c_2 \},$$
and the projections $p_1$ and $p_2$ are defined by $p_1(a \otimes c)= a \epsilon (c)$ and $p_2(a \otimes c)= \epsilon (a) c$, for any $a \otimes c \in A \times_B C$.

The kernel of a morphism $f \colon A\to B$ in $\Hopf_{K, coc}$ is given by the inclusion $\HKer(f)  \rightarrow A$ of the Hopf subalgebra
$$\HKer(f) = \{ a \in A \, \mid \, f(a_1) \otimes a_2 = 1_B \otimes a \} { = \{ a \in A \, \mid \, a_1 \otimes f(a_2) =  a \otimes  1_B\}.}$$
 of $A$. Given any coalgebra $C$, we shall write $$C^+ = \{x \in C \mid \epsilon (x) = 0\}.$$ 
 Remark that for any coalgebra morphism $f:C\to D$, we have that 
$f(C^+)=f(C)^+$. Indeed, $y\in f(C)^+$ iff $y=f(x)$ for some $x\in C$ and $0=\epsilon(y)=\epsilon(f(x))=\epsilon(x)$ iff $x\in C^+$ and $y=f(x)$.

 The cokernel of an arrow $f \colon A \rightarrow B$ in $\Hopf_{K,coc}$ is given by the canonical quotient 
$$q \colon B \rightarrow B/Bf(A)^+B,$$
where $f(A) = \{ f(a) \, \mid \, a \in A \}$ is the direct image of $A$ along $f$, and $B f(A)^+ B$ is 
the twosided $B$-ideal generated by $f(A)^+$, which can be checked to be
a Hopf ideal (i.e.\ a two-sided ideal and two-sided coideal that is stable under the antipode). For more details about basic properties of Hopf algebras we refer to \cite{Sweedler, A}.

 Let $A$ be a Hopf algebra then for a left $A$-module $M$ we will denote the action of an element $a\in A$ on $m\in M$ by
$${}^am.$$
Since $A$ is a Hopf algebra, the category of left $A$-modules is closed monoidal with a strict closed monoidal forgetful functor to $K$-vector spaces, where the left $A$-action on the tensor product $M\ot N$ of two left $A$-modules $M$ and $N$ is given by
$$ {}^a (m\ot n)={}^{a_1}m\ot {}^{a_2}n$$
for all $m\ot n\in M\ot N$. The monoidal unit in the category of left $A$-modules is the ground field $K$, endowed with a left $A$-module structure by means of the counit $\epsilon:A\to K$. 
A left $A$-module coalgebra is a coalgebra (comonoid) object in the category of left $A$-modules. More explicitly, $C$ is a left $A$-module coalgebra if $C$ is a left $A$-module which is at the same time a coalgebra, whose comultiplication and counit are morphisms of left $A$-modules. In Sweedler notation, these last two conditions mean that
$$ \Delta_C({}^a c)={}^{a_1} c_1\ot {}^{a_2} c_2,\quad \epsilon_C({}^a  c)=\epsilon_A(a)\epsilon_C(c)$$
for all $a\in A$ and $c\in C$.

The following result, due to Newman \cite{New}, will play a central role in what follows. We shall adopt the formulation given in \cite{Schneider}. When $I$ is a left ideal of a Hopf algebra $A$ that is also a (two-sided) coideal, the quotient $A/I$ has the structure of a left $A$-module coalgebra,
and the canonical projection $\pi:A\to A/I$ is a morphism of $A$-module coalgebras. 
Conversely, given any surjective morphism $\pi:A\to B$ where $A$ is a cocommutative Hopf algebra and $B$ is a (cocommutative) $A$-module coalgebra, the vector space kernel of $\pi$ naturally has the structure of a left ideal and (two sided) coideal. 

\begin{theorem}\label{Newman}\cite{New}
Let $A$ be a cocommutative Hopf algebra on a field $K$. Then there is a bijective correspondence between $${\mathcal S}= \{ D \subset A \, \mid \, D \, \mathrm{ \, is \, a \, Hopf \, subalgebra \, of \,} A \} $$ and 
  $${\mathcal I}= \{ I  \subset A \, \mid \, I \, \mathrm{ \, is \, a \, left \, ideal  \, and \, a \,  two \, sided \, coideal \, of \,} A \}.$$ 
\begin{enumerate}
\item
The correspondence $\Phi{_A} \colon {\mathcal S} \rightarrow {\mathcal I}$ associates, with a Hopf subalgebra $D$ of $A$, the corresponding left ideal two-sided coideal given by $\Phi{_A} (D)= AD^+$ (which is 
 the left $A$-module generated by $D^+$); 
\item Conversely, given a left ideal two-sided coideal $I$ in $A$, the inverse correspondence $\Psi{_A} \colon {\mathcal I} \rightarrow {\mathcal S}$ sends $I$ to the Hopf subalgebra of $A$ given by 
$$\Psi{_A}(I) = 
\{ x \in A \mid (id \otimes \pi) \Delta (x) = x \otimes \overline{1}\},$$
where $\pi$ is the canonical quotient $\pi : A \rightarrow A/I$ and $\overline{1}= \pi (1)$ is the equivalence class of the unit $1$ of the Hopf algebra $A$. 
\end{enumerate}
When $A$ is a cocommutative Hopf algebra, $B$ a left $A$-module cocommutative coalgebra and $f:A\to B$ a surjective morphism of $A$-module coalgebras we find that the vector space kernel $\ker (f)$ is a left ideal two-sided coideal in $A$, and moreover $\ker (f)=A\Psi{_A}(\ker (f))^+$.
\end{theorem}

{
Recall that a Hopf subalgebra $B\subset A$ of a cocommutative Hopf algebra $A$ is called {\em normal} if for any $b\in B$ and $a\in A$, one has that $a_1bS(a_2)\in B$. The following corollary is well-known, but we include a proof for sake of clarity.
\begin{corollary}\label{normal}
For a Hopf subalgebra $B\subset A$ of a cocommutative Hopf algebra $A$, the following conditions are equivalent:
\begin{enumerate}
\item $B$ is a normal Hopf subalgebra;
\item the associated ideal $\Phi{_A}(B)$ is a Hopf ideal (hence the quotient $A/\Phi{_A}(B)$ is a Hopf algebra);
\item  the inclusion morphism $B \rightarrow A$ is a \emph{normal monomorphism}, i.e. the kernel of some morphism in $\Hopf_{K, coc}$.
\end{enumerate}
In other words, Newman's correspondence can be restricted to a correspondence between normal Hopf subalgebras and Hopf ideals.
\end{corollary}

\begin{proof}
$(1)\Rightarrow (2).$
Let $B$ be a normal Hopf subalgebra of $A$, and consider the associated left ideal, two-sided co-ideal $\Phi{_A}(B)=AB^+$. Let us show that it is also a right ideal. For any $ab\in AB^+$ (sum implicitly understood) and any $a'\in A$, we find that
$$(ab)a'= a\epsilon(a'_1)ba'_2 = aa'_1S(a'_2)ba'_3\in AB^+.$$
To see that $AB^+$ is stable under the antipode, remark that $S(ab)=S(b)S(a)$. Since $b\in B^+$ and $B$ is stable under the antipode, $S(b)\in B^+\subset AB^+$ and therefore $S(b)S(a)\in AB^+$ since we have just shown this is a right ideal. 

$(2)\Rightarrow (3).$
If $I=\Phi{_A}(B)$ is a Hopf ideal, then the quotient $A/\Phi{_A}(B)$ is a Hopf algebra. By Newman's correspondence we know that $B=\Psi{_A}(I)$. Observe that $\Psi{_A}(I)$ is exactly the kernel $\HKer(\pi)$ of the quotient $\pi:A\to A/I$.

$(3)\Rightarrow (1).$
Suppose that $B=\HKer (f)$ for some morphism $f:A\to C$ in $\Hopf_{K,coc}$. Then we know that 
$x\in\HKer (f)$ if and only if {$f(x_1) \ot x_2= 1 \ot x$. For any $a\in A$, we then find that
\begin{eqnarray*}
f(a_1x_1S(a_4)) \ot a_2x_2S(a_3) & = & f(a_1S(a_4)) \ot a_2x S(a_3)  \\ &=& f(a_1S(a_2)) \ot a_3x S(a_4) \\ &=& 1 \ot a_1xS(a_2),  \end{eqnarray*}
and therefore $a_1xS(a_2)\in \HKer (f)$.}
\end{proof}
}

\begin{lemma}\label{first conditions}
The category $\Hopf_{K, coc}$ of cocommutative Hopf algebras over a field $K$ satisfies condition $(1)$ and $(2)$ in Lemma \ref{regularity}.
\end{lemma}
\begin{proof}
As remarked in \cite{GKV}, given a morphism $f \colon A \rightarrow B$, its regular epimorphism-monomorphism factorization $i \cdot p$ in $\Hopf_{K, coc}$ is obtained by taking the cokernel $p$ of the kernel $k$ of $f$:
$$
\xymatrix{{\HKer}(f) \ar[r]^-k  & A \ar[dr]_p \ar[rr]^f & & B \\
&  & {\frac{A}{A{{(\HKer}(f))^+ A}} } \ar[ur]_i  &
}
$$

Indeed, notice that the vector space kernel $\ker(f)$ is a Hopf ideal when $f : A \rightarrow B$ is a Hopf algebra morphism, then thanks to Newman's theorem and its Corollary \ref{normal}, we know that $\ker(f) = A(\HKer(f))^+$ which is equal to $A(\HKer(f))^+A$ since $\HKer(f)$ is normal. Therefore, we have that
$${\frac{A}{A{{(\HKer}(f))^+ A}} } = \frac{A}{\ker(f)}  \cong f(A),$$ 
and the epi-mono factorization in $\Hopf_{K, coc}$ is obtained as the usual one in vector spaces.

To see that condition $(2)$ is satisfied, observe that in $\Hopf_{K, coc}$ the regular epimorphisms are the same as surjective morphisms \cite{Chirva}, and products are tensor products: $E\times A = E \otimes A$. Accordingly, the induced arrow $1_E \times f \colon E \times A \rightarrow E \times B$ is surjective whenever $f$ is surjective.
\end{proof}
As explained in \cite{Yanagihara2} (Corollary $2$) given a morphism $p \colon A \rightarrow B$ in $\Hopf_{K, coc}$ and a Hopf subalgebra $C \subseteq B$ of $B$, then the subset of $A$ defined by $$p^{-1}(C) = \{ x \in A \mid [ (p \otimes id_A) \Delta (x) - 1 \otimes x ] \in C^+ \otimes A \}$$
is a Hopf subalgebra of $A$. This subalgebra is called the \emph{h-inverse} of $C$ (along $p$). 
\begin{lemma}\label{order-preserving}
Consider a morphism $p \colon A \rightarrow B$ in $\Hopf_{K, coc}$. Then:
\begin{enumerate}[(i)]
\item
for all Hopf subalgebras $C$ of $B$, $p (p^{-1} (C)) \subseteq C$; 
\item
for all Hopf subalgebras $D$ of $A$, $D\subseteq p^{-1} (p (D))$.
\end{enumerate}
\noindent
In other words, the $h$-inverse and the direct image along $p$ define an order preserving correspondence between the lattices of Hopf subalgebras of $A$ and $B$. 

Consequently,
\begin{enumerate}[(i)]
\setcounter{enumi}{2}
\item for all Hopf subalgebras $C\subseteq B$ and $D\subseteq A$,
$$D\subseteq p^{-1}(C) \Leftrightarrow p(D)\subseteq C;$$
\item for all Hopf subalgebras $C\subseteq B$
$$C=p (p^{-1} (C)) \Leftrightarrow C=p(D), {\text{ for some $D\subseteq A$}}.$$
\end{enumerate}

\end{lemma}

\begin{proof}
${(i)}$
 If $x \in p^{-1} (C)$, then $(p(x_1) \otimes x_2 - 1_C \otimes x) \in C^+ \otimes A$. By applying $id \otimes \epsilon$ to this element we get $p (x) - \epsilon (x)1_C \in C^+$, so that $p (x) \in C$.\\
${(ii)}$
First remark that $p(d)-\epsilon_D(d)1_B\in p(D)^+$, for any $d\in D$. Then we have 
\begin{eqnarray*}
p(d_1)\ot d_2-1_B\ot d&=&p(d_1)\ot d_2- \epsilon_D(d_1)1_B\ot d_2\\
&=&(p(d_1)-\epsilon_D(d_1)1_B)\ot d_2\in p(D)^+\ot A.
 \end{eqnarray*}
The two last statements follow since the direct and $h$-inverse image along $p$ define a pair of adjoint functors $$\xymatrix@=30pt{
{\mathsf{Sub}(A)  \, } \ar@<1ex>[r]_-{^{\perp}}^-{p^{-1}} & {\, \mathsf{Sub}(B) \, }
\ar@<1ex>[l]^-p  }
 $$ between the poset categories $\mathsf{Sub}(A)$ and $\mathsf{Sub}(B)$ of subobjects of $A$ and $B$, respectively. Let us make this more explicit. \\
$(iii)$
If $p(D)\subset C$ then it follows by the previous observations that 
$$D\subset p^{-1}(p(D))\subset p^{-1}(C).$$
The other implication is proven in the same way.\\

$(iv)$
In case $C=p(p^{-1}(C))$, we can take $D=p^{-1}(C)$. On the other hand, if $C=p(D)$ for some Hopf subalgebra $D$ of $A$ then by part $(ii)$ we have $D\subseteq p^{-1}(p(D))=p^{-1}(C)$. By applying $p$ one gets $C = p(D) \subseteq p(p^{-1}(C))$, and by part $(i)$ we obtain the desired equality.

\end{proof}

\noindent
 Let us now observe that the $h$-inverse $p^{-1}(C)$ is nothing but
the ``object part'' of the pullback of the inclusion morphism $i \colon C \rightarrow B$ along $p \colon A \rightarrow B$:
\begin{lemma}\label{inverse-p.b.}
Given an inclusion of a Hopf subalgebra  $i \colon C \rightarrow B$ of a cocommutative Hopf algebra $B$, then the diagram 
\begin{equation}\label{pullback-rho}
\xymatrix{p^{-1}(C)
\ar[r]^-{\hat{p}} \ar[d]_{j} & C \ar[d]^i \\
A \ar[r]_{p} & B.
}
\end{equation} 
is a pullback, where $j$ is the inclusion of $p^{-1}(C)$ in $A$, and $\hat{p}$ the restriction of $p$ to $p^{-1}(C)$.
\end{lemma}
\begin{proof}
As it follows from Lemma \ref{order-preserving} (i), for any $x$ in $p^{-1}(C)$ the element $p(x)$ is in $C$, and the diagram \eqref{pullback-rho} is commutative. To check the universal property, consider two morphisms $\alpha \colon T \rightarrow A$ and $\beta \colon T \rightarrow C$ such that $ p \cdot \alpha = i \cdot \beta$, and let us show that $\alpha (t) \in p^{-1}(C)$. We have the equalities: \begin{eqnarray}
(p \otimes {id}_A)  \Delta (\alpha (t) ) - 1 \otimes \alpha (t) & =&  (p \cdot \alpha \otimes \alpha )(t_1 \otimes t_2) - \epsilon (t_1) 1 \otimes \alpha (t_2) \nonumber \\
& = & i \cdot \beta (t_1) \otimes \alpha (t_2) - \epsilon (t_1)1 \otimes \alpha (t_2) \nonumber \\
& = & [i \cdot \beta (t_1) - \epsilon (t_1) 1 ]\otimes \alpha (t_2), \nonumber   \end{eqnarray}
Since $\beta (t) -\epsilon_T(t) 1\in C^+$ for any $t\in T$, we obtain that 
 $\alpha (t) \in p^{-1}(C)$, as desired.
 
\end{proof}

\begin{lemma}\label{normal2}
When $p \colon A \rightarrow B$ is a surjective Hopf algebra morphism in $\Hopf_{K, coc}$. Then for any Hopf subalgebra $D$ of $A$, we have that
\begin{enumerate}[(i)]
\item if $D$ is normal in $A$, its direct image $p (D)$ is a normal Hopf subalgebra of $B$,
\item taking the direct image commutes with applying the correspondence $\Phi$ from Theorem \ref{Newman}, i.e.
$$ \Phi_B \cdot p(D)=p \cdot \Phi_A(D).$$
\end{enumerate}
\end{lemma}
\begin{proof}
$(i)$
Assume that $D$ is a normal Hopf subalgebra of $A$, so that $a_1 d S(a_2) \in D$, for any $a \in A$ and any $d \in D$. Since $p$ is surjective, for any $b$ in $B$ there exists $a$ in $A$ such that $p(a) = b$ (and then $b_1 \otimes b_2 = p(a_1) \otimes p(a_2)$). Hence, for any $d$ in $D$, \begin{align*}
 b_1p(d)S(b_2) &= p(a_1)p(d)S(p(a_2))\\
 &= p(a_1dS(a_2)) \in p(D),
\end{align*} 
since $a_1dS(a_2) \in D$ (since $D$ is normal in $A$). Thus, $p(D)$ is a normal Hopf subalgebra of $B$. \\

$(ii)$ Making use of surjectivity of $p$ in the third equality, it follows that
$$p\cdot \Phi_A(D) = p(AD^+) = p(A)p(D^+) = Bp(D^+)=Bp(D)^+=\Phi_B\cdot p(D).$$

\end{proof}

The following result 
proves that the pullback of an epimorphism along a monomorphism in $\Hopf_{K,coc}$ is again an epimorphism. To prove this, we use the observation from Lemma \ref{inverse-p.b.} that the pullback along a monomorphism in $\Hopf_{K,coc}$ can be computed by means of the $h$-inverse.

\begin{proposition}\label{surjective}
Consider a surjective morphism $p \colon A \rightarrow B$ in $\Hopf_{K, coc}$ and a Hopf subalgebra $C$ of $B$, with inclusion $i \colon C \rightarrow B$. Then the morphism $\hat{p}$ in the following pullback is also surjective.
\begin{equation*}
\xymatrix{p^{-1} (C) \ar[d]_{j} \ar[r]^-{\hat{p}} \ar[d]_{} & C \ar[d]^i \\
A \ar[r]_{p} & B.
}\end{equation*}
\end{proposition}

\begin{proof}
From Lemma \ref{inverse-p.b.}, we know that $\hat{p}$ is just given by the restriction of $p$.
Hence $\hat p$ is surjective if and only if $C= p (p^{-1} (C))$. By Lemma \ref{order-preserving}(iv), it is equivalent to prove that $C=p(D)$ for some Hopf subalgebra $D$ of $A$. To construct this algebra $D$, consider the quotient $B/BC^+$, which is a left $B$-module coalgebra and thus a left $A$-module coalgebra by restriction of scalars via $p$. 
Consider now the following diagram in the category $\Vect$ of vector spaces:
$$\xymatrix{ 
A \ar[dr]_{\pi \cdot p } \ar[r]^{p} & B \ar[d]^{\pi} \\
 & {B/BC^+.}
}$$
Since $\pi$ and $p$ are both surjective, we have
\begin{equation}\label{surj2}
p(\ker(\pi\cdot p))=\ker(\pi)=BC^+. 
\end{equation}
Furthermore, $\pi\cdot p$ is an $A$-module coalgebra morphism, and therefore the (vector space) kernel of $\pi\cdot p$ is a left ideal and a two-sided coideal in $A$. 
 Using the correspondence of Theorem \ref{Newman}, we obtain therefore a Hopf subalgebra $D:=\Psi{_A}(\ker(\pi\cdot p))$ of $A$. To show that $C=p(D)$, let us observe that
$$\Phi_B(C)=\ker(\pi)=p(\ker(\pi\cdot p))=p\cdot \Phi_A\cdot \Psi_A(\ker(\pi\cdot p))=p\cdot \Phi_A(D)=\Phi_B\cdot p(D)$$
where we used \eqref{surj2} in the second equality, the Newman correspondence from Theorem \ref{Newman} in the third equality, and Lemma \ref{normal2}(ii) in the last equality. Applying once more the Newman correspondence, we find that 
$C= \Psi_B \cdot \Phi_B(C)= \Psi_B \cdot \Phi_B (p(D))=p(D)$ and this completes the proof.
\end{proof}

\begin{corollary}\label{regular}
The category $\Hopf_{K, coc}$ is regular. 
\end{corollary}

\begin{proof}
In any pullback \eqref{pullback} where $f \colon A \rightarrow B$ is a monomorphism, and $p \colon E \rightarrow B$ a regular epimorphism, we know that $p_2 \colon E \times_C B \rightarrow C$ is surjective by Proposition \ref{surjective}. By Lemma \ref{regularity} and Lemma \ref{first conditions} the proof is then complete.
\end{proof}

\begin{theorem}\label{semi-abelian}
The category $\Hopf_{K, coc}$ is semi-abelian.
\end{theorem}
\begin{proof}
The category $\Hopf_{K, coc}$ is pointed, with zero objet the base field $K$, and it is known to be protomodular \cite{Bourn}, since it can be seen as the category of internal groups in the category of cocommutative $K$-coalgebras. $\Hopf_{K, coc}$ is also (finitely) complete and cocomplete, as it follows immediately from the fact that it is a locally finitely presentable category \cite{Porst}. By Corollary \ref{regular} the category $\Hopf_{K, coc}$ is regular, so that it will be semi-abelian provided that the direct image of a normal monomorphism is again a normal monomorphism (see 3.7 in \cite{JMT}, for instance,  where this property is one of the so-called ``old axioms'' defining a semi-abelian category). This property is true by Lemma \ref{normal2} (i) (and it was also observed in \cite{Yanagihara2,VW}, for instance).
\end{proof} 

From this theorem one can easily deduce Takeuchi's theorem concerning the category $\Hopf_{K, comm, coc}$ of commutative and cocommutative Hopf algebras:

\begin{theorem} \cite{Takeuchi}
The category $\Hopf_{K, comm, coc}$ of commutative and cocommutative Hopf algebras over a field $K$ is an abelian category. 
\end{theorem}

\begin{proof}
It is well known that the category of abelian objects in a semi-abelian category is abelian (see Corollary $4.2$ in \cite{Gran}, for instance). It turns out that the category $\mathsf{Ab} ( \Hopf_{K, coc})$ of abelian objects in $\Hopf_{K, coc}$ is isomorphic to the category $\Hopf_{K, comm, coc}$ of commutative and cocommutative Hopf algebras. Indeed, it follows from Proposition $9$ in \cite{BournAbelian} that the abelian objects in any semi-abelian category $\mathbb C$ are precisely the objects $X$ with the property that the diagonal $\Delta \colon  X \rightarrow X \times X$ is a normal monomorphism ( = a kernel of some morphism in $\mathbb C$).
Accordingly, as observed in \cite{VW}, the abelian objects in $\Hopf_{K, coc}$ are precisely the cocommutative Hopf algebras $X$ such that the comultiplication $\Delta \colon X \rightarrow X \otimes X$ is a normal monomorphism or, equivalently, such that $X$ is commutative. 
\end{proof}

We finish this section by an immediate corollary that, similarly to the above, can be viewed as a semi-abelian extension of the result due to Grothendieck which states the category of finite dimensional, commutative and cocommutative Hopf algebras is abelian.

\begin{corollary}
The category $\Hopf_{K, coc}^{fd}$ of finite dimensional cocommutative Hopf algebras over a field $K$ is semi-abelian.

The opposite category $(\Hopf_{K, comm}^{fd})^{op}$ of the category $\Hopf_{K, comm}^{fd}$ of finite dimensional commutative Hopf algebras over a field $K$ is semi-abelian.
\end{corollary}

\begin{proof}
The first statement follows immediately from Theorem \ref{semi-abelian} along with the observation that $\Hopf_{K, coc}^{fd}$ is a full subcategory of $\Hopf_{K, coc}$ which is closed under finite limits and regular quotients.

The second statement follows from the fact that the vector space dual induces an equivalence between the categories  $\Hopf_{K, coc}^{fd}$ and $(\Hopf_{K, comm}^{fd})^{op}$.
\end{proof}

\section{Action representability of $\Hopf_{K, coc}$}
Recall that a pointed protomodular category \cc is action representable (equivalently, that it has representable object actions in the sense of \cite{BJK}) if the following property holds: 
given any object $X$ in $\cc$,  
there is a split extension  
\begin{equation}\label{splitgeneral}
\xymatrixcolsep{2pc}\xymatrix{
0 \ar[r]^-{} & X  \ar[r]^{i_1}  & \overline{X}  \ar@<-0.5ex>[r]_{p_2} & [X] \ar@<-0.5ex>[l]_{i_2} \ar[r] & 0 }\end{equation}
with kernel $X$
that is universal in the sense that, given any other split extension in \cc with kernel $X$ 
$$
\xymatrixcolsep{2pc}\xymatrix{
 0 \ar[r] & X  \ar[r]^k  & A \ar@<-0.5ex>[r]_-{f} & B \ar@<-0.5ex>[l]_-{s} \ar[r] & 0 }$$
there is a unique (up to isomorphism) $\chi \colon B \rightarrow [X]$ (and $\overline{\chi}\colon A\longrightarrow \overline{X}$) such that the following diagram of split exact sequences commutes \[\xymatrixcolsep{2pc}\xymatrix{
0 \ar[r]^-{}  & X \ar@{=}[d] \ar[r]^-{k}  & A \ar@<-0.5ex>[r]_-{f} \ar[d]^{\overline{\chi}} & B \ar@<-0.5ex>[l]_-{s} \ar[r] \ar[d]^{\chi} & 0
 \\
0 \ar[r]^-{} & X  \ar[r]^-{i_1}  & \overline{X} \ar@<-0.5ex>[r]_-{p_2} & [X] \ar@<-0.5ex>[l]_-{i_2} \ar[r] & 0.
}\] 
The object $[X]$ in \eqref{splitgeneral} is called a \textit{split extension classifier} for $X$. In the action representable category $\mathsf{Grp}$ of groups the object $[X]$ is simply given by the group $\mathsf{Aut}(X)$ of automorphisms of the group $X$. In the category $\mathsf{Lie_K}$ of Lie algebras over a field $K$ the split extension classifier $[A]$ of a Lie algebra $A$ is given by the Lie algebra $\mathsf{Der}(A)$ of derivations of $A$.

\begin{proposition}\label{action-rep}
The category $\Hopf_{K, coc}$ is action representable.
\end{proposition}
\begin{proof}
By Theorem $4.4$ in \cite{BJK} one knows that the category of internal groups in a cartesian closed category is always action representable, provided it is semi-abelian. The result then follows from Theorem \ref{semi-abelian} and the observation that the category $\Hopf_{K, coc}$ of cocommutative Hopf algebras is the category of internal groups in the cartesian closed category $\mathsf{Coal}_{K, coc}$ of cocommutative coalgebras (see \cite{GP} for the fact that $\mathsf{Coal}_{K, coc}$ is a cartesian closed category).
\end{proof}
 \begin{remark}
The same argument as in the proof of the Proposition here above implies that $\Hopf_{K, coc}$ is locally algebraically cartesian closed in the sense of \cite{Gray}, and then algebraically coherent in the sense of \cite{CGV}.
\end{remark}
\begin{remark}
An explicit description of the split extension classifier $[X]$ of a cocommutative Hopf algebra in $\Hopf_{K, coc}$ was given in \cite{GKV2} in the special case when the characteristic of the base field $K$ is zero, by using the canonical semi-direct product decomposition in the Milnor-Moore Theorem \cite{MM}. It would be interesting to give such a description in the case of a general field $K$.
\end{remark}

\section{Commutators in $\Hopf_{K, coc}$}
In any pointed category $\mathbf{C}$ with binary products, one says that two subobjects $x \colon X\rightarrow A $ and $y \colon Y \rightarrow A$ of the same object $A$ \emph{commute} (in the sense of Huq) \cite{Huq} if and only if there exists an arrow $p$ making the following diagram commute:

\begin{equation}\label{connector}
\begin{tikzpicture}[descr/.style={fill=white},scale=0.9]
\node (A) at (0,0) {$A$};
\node (C) at (2,2) {$Y$};
\node (D) at (0,2) {$X \times Y$};
\node (B) at (-2,2) {$X$};
  \path[-stealth]
 (B.south) edge node[left] {${x\,} $} (A.north west) 
 (C.south) edge node[right] {${\, y} $} (A.north east) 
 (B.east) edge node[above] {$(1,0)$} (D.west) 
 (C.west) edge node[above] {$(0,1)$} (D.east);
   \path[-stealth,dashed]
 (D.south) edge node[right] {$p$} (A.north);
\end{tikzpicture}
\end{equation}
Since $\mathsf{Hopf}_{K,coc}$ is \emph{protomodular}, such an arrow $p$ is unique, when it exists \cite{BG}. By taking into account the fact that the categorical product $X\times Y$ in $\mathsf{Hopf}_{K,coc}$ is the tensor product $X \otimes Y$ we get:
\begin{lemma}\label{commutatorHopf}
In $\mathsf{Hopf}_{K,coc}$, for any two Hopf subalgebras $x \colon X \rightarrow A$, and $y \colon Y \rightarrow A$ of a cocommutative Hopf algebra $A$, the following conditions are equivalent : 
\begin{itemize}
\item[(a)] there exists a unique morphism of Hopf algebras $p : X \otimes Y \rightarrow A$ such that diagram \eqref{connector} commutes;
\item[(b)]  $ab=ba$, $\forall a \in X$ and $\forall b \in Y$;
\item[(c)] $a_1b_1S(a_2)S(b_2) = \epsilon(a)\epsilon(b)$, $\forall a \in X$ and $\forall b \in Y$.
\end{itemize}
\end{lemma}
\begin{proof}
$(a) \Leftrightarrow (b)$. First note that, in $\mathsf{Hopf}_{K,coc}$, if there is a $p$ making diagram \eqref{connector} commute then $p$ has to be defined as $p(a \otimes b) = ab$: indeed,
\begin{align*}
p(a \otimes b) &= p((a \otimes 1)(1 \otimes b)) \\ &= p(a \otimes 1) p (1 \otimes b) \\
&= ab.
\end{align*}  
Then note that, if $(b)$ holds, then $p$ is an algebra morphism (and it is always a coalgebra morphism by the cocommutativity assumption). 

On the other hand, when $(a)$ holds and $p$ is an algebra morphism, 
$$ ab = p((a \otimes 1)(1 \otimes b)) = p((1 \otimes b)(a \otimes 1)) = ba.$$
$(b) \Leftrightarrow (c)$ When condition $(b)$ holds, then we have the following identities: $$ a_1b_1S(a_2)S(b_2) = a_1S(a_2)b_1S(b_2) = \epsilon(a)\epsilon(b).$$ 
Conversely, $a_1b_1S(a_2)S(b_2) = \epsilon(a)\epsilon(b)$ for all $a \in X$ and $b \in Y$ implies that 
\begin{eqnarray*}
ab&=&a_1b\epsilon(a_2)=a_1bS(a_2)a_3=a_1b_1S(a_2)\epsilon(b_2)a_3 \\
&=&a_1b_1S(a_2)S(b_2)b_3a_3 = \epsilon(a_1)\epsilon(b_1)a_2b_2 = ba.
\end{eqnarray*}
 \end{proof}

In a general semi-abelian category, the Huq commutator of two normal subobjects $x \colon X\rightarrow A $ and $y \colon Y \rightarrow A$ is the smallest normal subobject $K \rightarrow A$ such that its cokernel $q \colon A \rightarrow A/K$ has the property that the images $q(X)$ and $q(Y)$ by $q$ commute in the quotient:
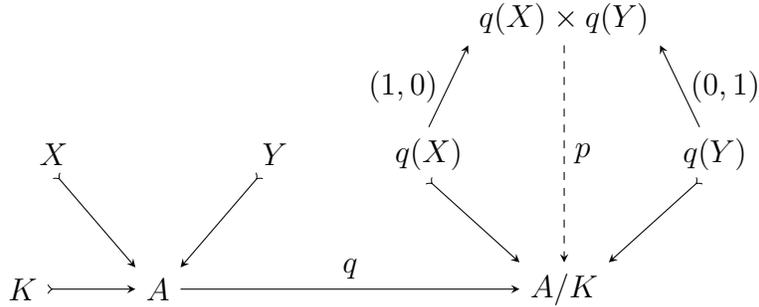
\begin{center}
\begin{tikzpicture}[descr/.style={fill=white},scale=0.9]
\node (A) at (0,0) {$K$};
\node (B) at (2,0) {$A$};
\node (C) at (0.5,2) {${X \,}$ };
\node (D) at (3.5,2) {${\quad Y}$};
\node (E) at (8,0) {$A/K$};
\node (F) at (6,2) {$q{(X)}$};
\node (G) at (10,2) {$\quad q(Y)$};
\node (H) at (8,4) {$q{(X)} \times q{(Y)}$};
\path[>-stealth]
 (A.east) edge node[right] {$ $} (B.west) 
 (C.south) edge node[above] {$ $} (B.north west) 
 (D.south) edge node[above] {$ $} (B.north east)
    (F.south) edge node[above] {$ $} (E.north west)
   (G.south) edge node[above] {$ $} (E.north east); 
  \path[-stealth]
 (B.east) edge node[above] {$ q $} (E.west) 
  (F.north) edge node[left] {$(1,0) $} (H.south west)
   (G.north) edge node[right] {$(0,1) $} (H.south east);
   \path[-stealth,dashed]
 (H.south) edge node[right] {$p$} (E.north);
\end{tikzpicture}
\end{center}

In particular, in the category $\mathsf{Hopf}_{K,coc}$, the Huq commutator of two normal Hopf subalgebras $x \colon X\rightarrow A $ and $y \colon Y \rightarrow A$, denoted by $[X,Y]_{\mathsf{Huq}}$, is then defined as the smallest normal Hopf subalgebra in $A$ such that, in the quotient $A/A[X,Y]_{\mathsf{Huq}}^+$,
the normal Hopf subalgebras $q(X)$ and $q(Y)$
commute  in the sense of Huq, where $q:A\to A/A[X,Y]_{\mathsf{Huq}}^+$ is the canonical projection. By Lemma \ref{commutatorHopf} this is also equivalent to the condition  $$\overline{ab} = \overline{ba}$$ $\forall a \in X, b \in Y$ or, equivalently, to $$ab -ba \in A[X,Y]^+_{\mathsf{Huq}}.$$

We shall now give an explicit description of the Huq commutator of two normal Hopf subalgebras $X$ and $Y$ of $A$, and show that it coincides with the commutator defined in \cite{Yanagihara}, compare also to the commutator subalgebra defined in \cite{Burciu} and studied in \cite{Cohen}, \cite{Cohen2}.
We write $[X,Y]$ for the subalgebra of A generated by all elements of the form
$$\{a,b \} = a_1b_1S(a_2)S(b_2)$$
for any $a \in X$ and any $b \in Y$.

\begin{proposition}
The algebra $[X,Y]$ is a normal Hopf subalgebra of $A$.
\end{proposition}

\begin{proof}
Thanks to the cocommutativity, $[X,Y]$ is a subcoalgebra and hence a subbialgebra  of $A$. Indeed,
\begin{align*}
\Delta(\{a,b\}) &= \Delta(a_1b_1S(a_2)S(b_2)) \\
& = a_1 b_1 S(a_4) S(b_4) \otimes a_2 b_2 S(a_3) S(b_3) \\
&= a_1b_1S(a_2)S(b_2) \otimes a_3b_3S(a_4)S(b_4) 
&\in [X,Y] \otimes [X,Y].
\end{align*}
Moreover, this subbialgebra is a Hopf subalgebra:
\begin{align*}
S(a_1b_1S(a_2)S(b_2)) &= b_2a_2S(b_1)S(a_1)
\\ &= b_1a_1S(b_2)S(a_2) \in [X,Y],
\end{align*}
where we used that $S^2=id$ since $A$ is cocommutative (see e.g.\ \cite[Proposition 4.0.1(6)]{Sweedler}).
Finally, given any $c$ in $A$, we check that the element $c_1a_1b_1S(a_2)S(b_2)S(c_2)$ belongs to $[X,Y]$ for any $a \in X$ and any $b \in Y$. Indeed, this is the case, since $X$ and $Y$ are normal Hopf subalgebras of $A$:
\begin{align*}
c_1a_1b_1S(a_2)& S(b_2)S(c_2) = c_1a_1S(c_2)c_3b_1S(c_4)c_5S(a_2) S(c_6) c_7 S(b_2)S(c_8) \\
&= (c_1aS(c_2))_1(c_3bS(c_4))_1S((c_1aS(c_2))_2)S((c_3bS(c_4))_2).
\end{align*}
\end{proof} 

\begin{proposition}\label{description}
Let $X,Y$ be two normal Hopf subalgebras of $A$. Then $$[X,Y] = [X,Y]_{\mathsf{Huq}}.$$
\end{proposition}
\begin{proof}
 If $q \colon A \rightarrow A/A[X,Y]^+$ denotes the canonical quotient, as a first step we will prove that $q(X)$ and $q(Y)$ commute in $A/A[X,Y]^+$. As we have already observed, in $\mathsf{Hopf_{K,coc}}$, this is equivalent to the condition $$ab-ba \in A[X,Y]^+,$$ for any $a \in X$ and any $b \in Y$.
Equivalently, this can be expressed by asking that $$ ab - ba = a_1b_1\Big(\epsilon(a_2)\epsilon(b_2)1 - S(b_2)S(a_2)b_3a_3\Big) \in A [X,Y]^+,$$
that holds by definition of $[X,Y]^+$.

Next, we will show that $[X,Y]$ is the smallest normal Hopf subalgebra which satisfies the definition of commutators. In other words, we are going to prove that $[X,Y]$ factorizes through the categorical kernel of any surjective Hopf algebra morphism for which the images of $X$ and $Y$ commute. Let $f \colon A \rightarrow B$ be any surjective Hopf algebra morphism  such that $f(X)$ and $f(Y)$ commute in $B$. \\
For any generator $\{a,b\}$ of $[X,Y]$, we have:
\begin{eqnarray*}
f(a_1b_1S(a_2)S(b_2)) \otimes a_3b_3S(a_4)S(b_4)\\
&&\hspace{-4cm}=  f(a_1) f(b_1)f(S(a_2))f(S(b_2))) \otimes a_3b_3S(a_4)S(b_4) \\
&&\hspace{-4cm}=  f(a_1) f(S(a_2))f(b_1)f(S(b_2))) \otimes a_3b_3S(a_4)S(b_4) \\
&&\hspace{-4cm}= 1 \otimes a_1b_1S(a_2)S(b_2).
\end{eqnarray*} 
 In conclusion, the Huq commutator  $[X,Y]_{\mathsf{Huq}}$ of two normal Hopf subalgebras $X$ and $Y$ of $A$ is $[X,Y]$.
\end{proof}
\begin{remark}\label{Huq=Smith}
It is well known that in any action representable category the notion of centrality of two equivalence relations $R$ and $S$ in the sense of Smith  is equivalent to the notion of centrality of the corresponding normal subobjects $N_R$ and $N_S$ in the sense of Huq \cite{BJ}. Accordingly, by taking into account Proposition \ref{action-rep}, in the category $\mathsf{Hopf}_{K,coc}$ the description of the commutator given in Proposition \ref{description} also applies to the (normalization of the) Smith commutator.
 \end{remark}

\section{Internal crossed modules of Hopf algebras}

The abstract notion of internal crossed module was defined and investigated in the general context of semi-abelian categories by Janelidze \cite{Jan}. Internal crossed modules in a semi-abelian category $\cc$ form a category that is equivalent to the category of internal groupoids in $\cc$. 

In the context of Hopf algebras, the notion of \emph{Hopf crossed module} was defined independently by Fern\'andez Vilaboa, L\'opez L\'opez, and Villanueva Novoa \cite{Vilaboa} (see also \cite{Majid, JFM}). We will show here that in the cocommutative case, the category $\mathsf{HXMod_{K, coc}}$ of Hopf crossed modules is equivalent to the category of internal groupoids in $\Hopf_{K, coc}$ and is consequently also equivalent to the category of internal crossed modules in $\Hopf_{K, coc}$.

Recall that for a cocommutative Hopf algebra $B$, the category of left $B$-modules is symmetric monoidal and the forgetful functor to vector spaces is symmetric monoidal. Hence, one can consider a Hopf algebra in the category of left $B$-modules, which is then called a $B$-module Hopf algebra. Explicitly, a (cocommutative) $B$-module Hopf algebra $X$ is a (cocommutative) Hopf algebra $X$ endowed with a linear map 
$ \xi \colon B \otimes X \rightarrow X,\ \xi (b \otimes x)  = \;^bx$, satisfying the following identities

\begin{eqnarray*}
&(a)& \;^{(bb')}x = \;^{b}(\;^{b'}x) \\
&(b)& \;^{1_B}x=x \\
&(c)& \;^bxy = \;^{b_1}x\;^{b_2}y \\
&(d) &\;^b1_X = \epsilon(b)1_X \\
&(e)& (\;^bx)_1 \otimes (\;^bx)_2 = \;^{b_1}x_1 \otimes \;^{b_2}x_2\\ 
&(f)& \epsilon(\;^bx) = \epsilon(b)\epsilon(x)
\end{eqnarray*}
for any $b, b' \in B$, and $x,y \in X$.\\
Remark that a $B$-module Hopf algebra is in particular a $B$-module coalgebra as defined in Section \ref{preliminaries}. We recall Majid's definition of Hopf crossed module in the context of $\mathsf{Hopf}_{K,coc}$ which, in this case, also coincides with the definition given in \cite{Vilaboa}.

\begin{definition}\cite{Majid}\label{Majid}
In $\mathsf{Hopf_{K, coc}}$, a \emph{Hopf crossed module} is a triple $(B,X,d)$ where $B$ is a cocommutative Hopf algebra, $X$ is a cocommutative $B$-module Hopf algebra and $d : X \rightarrow B$ is a Hopf algebra morphism satisfying
\begin{eqnarray*}
d(\;^bx) &=& b_1d(x)S(b_2) \\
\;^{d(y)}x &=& y_1xS(y_2)
\end{eqnarray*}
for any $x,y \in X$ and $b \in B$.
\end{definition}
The second axiom $$^{d(y)}x = y_1xS(y_2),$$ is usually called the \emph{Peiffer identity}.
Let $\mathsf{HXMod_{K, coc}}$ be the category of Hopf crossed modules, where a morphism $$(\alpha, \beta) \colon (B,X,d) \rightarrow (B',X',d')$$ is a pair of Hopf algebra morphisms $\alpha \colon X \rightarrow X'$ and $\beta \colon B \rightarrow B'$ such that $d' \cdot \alpha = \beta \cdot d$ and $\alpha (^{b}x ) = \, ^{\beta(b)} \alpha (x)$.\\
\\
When $A$ is a $B$-module Hopf algebra, one can define the semi-direct product $A \rtimes B$ (which is, in the Hopf algebra context, usually called the \emph{smash product} and denoted by $A\#B$) as the vector space $A \otimes B$ with the following structure maps:
\begin{align*}
(a \otimes b) (a' \otimes b')&= a \;^{b_1}a' \otimes b_2 b',\\
u_{A \rtimes B} &= u_A \otimes u_B,\\
\Delta_{A \rtimes B} (a \otimes b) &= a_1 \otimes b_1 \otimes a_2 \otimes b_2,\\
\epsilon_{A \rtimes B}(a \otimes b) &= \epsilon(a)\epsilon(b),\\
S_{A \rtimes B}(a \otimes b) &= \;^{S_B(b_1)}S_A(a) \otimes S_B(b_2), 
\end{align*}  
for any $a,a' \in A$ and any $b,b' \in B$.\\
This semi-direct product induces a well known correspondence between split epimorphisms and actions (see \cite{Molnar}). For the reader's convenience we now recall this correspondence in detail, in the cocommutative case, since it will be useful later on:
\begin{proposition}\label{iso crossed product}
 Let
$\xymatrix{
 A  \ar[rr]_{\delta} && B \ar@<-1 ex>[ll]_{i}}$
 be a split epimorphism in $\Hopf_{K,coc} $, i.e. $\delta \cdot i = Id_{B}$. Then, $\HKer(\delta)$ is a $B$-module Hopf algebra for the action $\xi \colon B \otimes \HKer(\delta) \rightarrow \HKer(\delta)$ defined by
$$\xi (b \otimes k) = i(b_1)ki(S(b_2)).$$  Moreover, there exists a natural isomorphism $$\HKer(\delta) \rtimes B \cong A.$$ 
\end{proposition}
\begin{proof}

Using the explicit description of the kernel of a morphism in $\Hopf_{K,coc} $ recalled above, one can check that $\xi (b \otimes k)\in \HKer\delta$ for all $b\ot k\in B\ot \HKer\delta$. Furthermore it is easy to verify that $\xi$ defines a $B$-module Hopf algebra structure on $\HKer(\delta)$.
On the other hand, one can check that the maps $\phi  \colon \HKer(\delta) \rtimes B \rightarrow A$ and $\psi \colon A \rightarrow \HKer(\delta) \rtimes B$ defined by 
$
\phi (k \otimes b) = ki(b)$ and $
  \psi (a) = a_1(i \cdot \delta)(S(a_2)) \otimes \delta(a_3)$
 are well-defined, and they are Hopf algebra morphisms.
 Moreover, these morphisms are mutual inverses.
\end{proof}
  
 Recall from \cite{CPP} that a \emph{reflexive-multiplicative graph} $\mathbb A$
 is a diagram \begin{equation}\label{graph}
\xymatrix{
A_1\times_{A_0}A_1 \ar[r]^-{m} & A_1 \ar@<1.3 ex>[rr]^{\delta} \ar@<-2.2 ex>[rr]_{\gamma} && A_0, \ar@<0.7 ex>[ll]_{i}}
\end{equation}
where $\delta$ is the ``domain morphism'', $\gamma$ the ``codomain morphism'', $i$ the ``identity morphism'' (so that $\delta \cdot i = \gamma \cdot i = Id_{A_0}$), $A_1\times_{A_0}A_1$ is the (object part of the) pullback
$$
\xymatrix{A_1\times_{A_0} A_1 \ar[d]_{p_1} \ar[r]^-{p_2} & A_1 \ar[d]^{ \gamma} \\
A_1 \ar[r]_-{\delta}& A_0
}
$$ and $m$ a multiplication that is required to satisfy the identities
\begin{equation}\label{RGM}
m \cdot (Id_{A_1} , i \cdot \delta ) = Id_{A_1} = m \cdot (i \cdot \gamma, Id_{A_1} ),
\end{equation} 
 where $(Id_{A_1} , i \cdot \delta ) : A_1 \rightarrow A_1\times_{A_0}A_1 $ and $(i \cdot \gamma, Id_{A_1} ) : A_1 \rightarrow A_1\times_{A_0}A_1 $ are induced by universal property of the pullback $A_1\times_{A_0}A_1$.
 
A reflexive-multiplicative graph is an \emph{internal groupoid} if the multiplication $m$ satisfies the additional identities
 \begin{equation*}
 \delta \cdot m = \delta \cdot {p_2},
 \end{equation*}
 \begin{equation*}
 \gamma \cdot m = \gamma \cdot {p_1},
 \end{equation*}
 \begin{equation*}
 m \cdot (1,m) = m \cdot (m,1),
 \end{equation*}
 and there exists a morphism $\iota \colon A_1 \rightarrow A_1$ such that \begin{equation*}
 \delta  \cdot \iota = \gamma,
 \end{equation*}
 \begin{equation*}
 \gamma \cdot \iota = \delta,
 \end{equation*}
 \begin{equation*}
 m \cdot (\iota , Id_{A_1}) = i \cdot {\delta},
 \end{equation*}
 \begin{equation*}
 m \cdot (Id_{A_1}, \iota) = i \cdot {\gamma}.
 \end{equation*}

 Reflexive-multiplicative graphs, with morphisms of reflexive graphs that preserve the multiplication, form a category, that will be denoted by $\mathsf{RMG}({\mathbb C})$. The category of internal groupoids will be denoted by $\mathsf{Grpd}(\mathbb{C})$. 
 
Recall that a finitely complete category $\mathbb{C}$ is a \emph{Mal'tsev} category when any internal reflexive relation in $\mathbb{C}$ is an equivalence relation. It is well known that any protomodular category (therefore in particular any semi-abelian category) is a Mal'tsev category (see \cite{BB}, for instance).
As shown in \cite{CPP}, when $\mathbb C$ is a Mal'tsev category, morphisms in $\mathsf{RMG}({\mathbb C})$ or in $\mathsf{Grpd}(\mathbb{C})$ are simply morphisms of reflexive graphs. 

\begin{remark}\label{remark-groupoid}
 It is well known \cite{CPP} that for a reflexive graph \begin{equation}\label{reflexive.graph}
\xymatrix{
 A_1 \ar@<1.3 ex>[rr]^{\delta} \ar@<-2.2 ex>[rr]_{\gamma} && A_0, \ar@<0.7 ex>[ll]_{i}}\end{equation}
 it is equivalent to be a reflexive multiplicative graph, an internal groupoid
or to have a double centralizing equivalence relation on the kernel pairs $\mathsf{Eq}(\delta)$ and $\mathsf{Eq}(\gamma)$ of $\delta$ and $\gamma$, respectively.
\end{remark}

In our context, as recalled in Remark \ref{Huq=Smith}, this condition of centralization for equivalence relations exactly corresponds to the commutation of the normal subobjects associated with $\mathsf{Eq}(\delta)$ and $\mathsf{Eq}(\gamma)$, i.e. $[\HKer(\delta),\HKer(\gamma)]_{\mathsf{Huq}} =0$.\\

We recall that a \emph{${cat}^1$-Hopf algebra}, in the sense of \cite{Vilaboa}, is a reflexive graph  \eqref{reflexive.graph} such that \begin{equation}
ab = ba, 
\end{equation}
  $\forall a \in \HKer(\delta)$ and $\forall b \in \HKer(\gamma)$. \\  We write $\mathsf{Cat^1}(\Hopf_{K, coc})$ for the category whose objects are ${cat}^1$- Hopf algebras and whose morphisms are the ones of reflexive graphs. \\

  In the context of cocommutative Hopf algebras, Remark \ref{remark-groupoid} leads us to the following proposition: 

\begin{proposition}\label{equivalence-cat}
In $\mathsf{Hopf}_{K,coc}$, for a reflexive graph \eqref{reflexive.graph} the following conditions are equivalent: 
 \begin{itemize}
  \item[(a)]  \eqref{reflexive.graph} is a reflexive-multiplicative graph;
 \item[(b)]  \eqref{reflexive.graph} is an internal groupoid;
 \item[(c)]   \eqref{reflexive.graph} satisfies $[\HKer(\delta),\HKer(\gamma)]_{\mathsf{Huq}} =0$;
 \item[(d)]  \eqref{reflexive.graph} is a ${cat}^1$-Hopf algebra.
 \end{itemize}
 This implies that the categories $\mathsf{RMG}(\Hopf_{K, coc})$, $\mathsf{Grpd}(\Hopf_{K, coc})$ and $\mathsf{Cat^1}(\Hopf_{K, coc})$ are isomorphic.
\end{proposition}
\begin{proof}
Since ${\mathbb C} = \Hopf_{K, coc}$ is a semi-abelian category (Theorem \ref{semi-abelian}), thus in particular a Mal'tsev category, this proposition follows from Remark \ref{remark-groupoid} and the description of the Huq commutator given in Proposition \ref{description}.
\end{proof}

We will now show that the categorical notion of crossed module in the category of cocommutative Hopf algebras corresponds exactly with the notion of Hopf crossed module, by using the fact that the both categories are equivalent to $\mathsf{Grpd}(\Hopf_{K, coc})$.
The next theorem could be derived from Theorem $14$ in \cite{Vilaboa}, but we prefer to give a sketch of the proof for sake of completeness.

\begin{proposition}\label{equivalenceMajid}
The categories $\mathsf{HXMod_{K, coc}}$ and $\mathsf{Grpd}(\Hopf_{K, coc})$ are equivalent.
\end{proposition}
\begin{proof}
Thanks to the Proposition above, it will suffice to prove that $\mathsf{RMG}(\Hopf_{K, coc})$ and $\mathsf{HXMod_{K, coc}}$ are equivalent. For brevity, we shall mainly recall the definition on objects of the functors yielding this equivalence. With a reflexive multiplicative graph \eqref{graph} one associates the Hopf algebra morphism $$ d = \gamma \cdot {hker} (\delta) \colon \HKer (\delta) \rightarrow A_0$$ 
(where ${hker} (\delta) \colon \HKer(\delta) \rightarrow A_1$ is the kernel of $\delta$), equipped with the action $A_0 \otimes \HKer(\delta) \rightarrow \HKer(\delta)$ defined by 
$$^a  k = i(a_1) k i(S(a_2)), $$
for any $a \in A_0$ and $k \in \HKer(\delta)$. 
Thanks to Proposition \ref{iso crossed product}, $\HKer(\delta)$ is an $A_0$-module Hopf algebra. The morphism $d = \gamma \cdot {hker} (\delta)$ is a Hopf algebra morphism by construction, so it remains to see that the two axioms in the Definition \ref{Majid} of Hopf crossed module are satisfied. 
The first axiom holds, since
\begin{eqnarray*}
 d(\;^ak)& = & [\gamma \cdot {hker} (\delta) ] (i(a_1) k i(S(a_2)))\\
 & =&  (\gamma \cdot i) (a_1)  (\gamma \cdot hker(\delta)) (k) (\gamma \cdot i) (S(a_2)) \\
  & =&  a_1 d (k) S(a_2).
\end{eqnarray*}
By Proposition \ref{equivalence-cat}, the elements of $\HKer(\delta)$ and $\HKer(\gamma)$ commute, hence for all $b,k \in \HKer(\delta)$, \begin{equation}\label{commutative} (S(k_1) (i \cdot \gamma)(k_2))b = b ((S(k_1) (i \cdot \gamma)(k_2))\end{equation} since $(S(k_1) (i \cdot \gamma)(k_2)) \in \HKer(\gamma)$, as it follows from
\begin{align*}
\gamma(S(k_1)(i\cdot \gamma)(k_2)) \otimes S(k_3)(i\cdot \gamma)(k_4) &= \gamma(S(k_1))\gamma(k_2) \otimes S(k_3)(i\cdot \gamma)(k_4)\\
&= 1_B \otimes S(k_1)(i \cdot \gamma)(k_2).
\end{align*}
Accordingly, the equality \eqref{commutative} implies 
\begin{align*}
^{d(k)}b &=(i \cdot \gamma)(k_1))b (i \cdot \gamma)(S(k_{2}))\\
&= k_1S(k_2)(i \cdot \gamma)(k_3))b (i \cdot \gamma)(S(k_{4}))\\
&=  k_1b S(k_2)(i \cdot \gamma)(k_3))(i \cdot \gamma)(S(k_{4}))\\
&= k_{1}bS(k_2)
\end{align*}
showing that the \textit{Peiffer identity} holds and $d \colon \HKer(\delta) \rightarrow A_0$ is indeed a Hopf crossed module in the sense of \cite{Majid}.\\
\\
Conversely, given a Hopf crossed module $d : X \rightarrow B$, we define the reflexive graph
\begin{equation}\label{reflexive}
\xymatrix{
 X \rtimes B \ar@<1.3 ex>[rr]^{{p_2}} \ar@<-2.3 ex>[rr]_{{p_1}} && B, \ar@<0.7 ex>[ll]_{e}}
 \end{equation}
The morphism $p_2$ is the second projection $p_2 (x \otimes b) = \epsilon (x) b$, the morphisms $p_1$ and $e$ are defined by $p_1(x \otimes b)= d(x)b$ and $e(b) = 1_X \otimes b$. This reflexive graph is equipped with a groupoid multiplication by setting $$m(x \otimes b, x' \otimes b') = xx' \otimes \epsilon (b) b',$$
 for any $(x \otimes b, x' \otimes b')$ in the pullback $P$ defined by 
 $$\xymatrix{P \ar[r]^-{\pi_2} \ar[d]_{\pi_1} &  X \rtimes B \ar[d]^{p_1} \\
 X \rtimes B \ar[r]_-{p_2} & B.} $$
 We leave the verification of the fact that $m$ gives a reflexive-multiplicative graph structure on \eqref{reflexive} to the reader. We observe that the Peiffer identity is essential to prove that the map $m \colon P \rightarrow X \rtimes B$ is an algebra morphism, and is then a morphism in $\Hopf_{K, coc}$.\\
\\
The correspondence described above naturally extends to morphisms, yielding two functors \begin{align*}
F &: \mathsf{RMG}(\Hopf_{K, coc}) \rightarrow \mathsf{HXMod_{K, coc}} ,\\
G &: \mathsf{HXMod_{K, coc}} \rightarrow \mathsf{RMG}(\Hopf_{K, coc}).
\end{align*}
These functors give rise to an equivalence of categories. \end{proof}

\begin{remark}
The previous result can also be deduced from the recent results of B\" ohm (see Proposition $3.13$ in \cite{Bohm}, where a Hopf monoid in $\mathsf{Vect}$ is precisely a Hopf algebra).
\end{remark}
Let us write $\mathsf{XMod}( \Hopf_{K, coc})$ for the category of internal crossed modules in $\Hopf_{K, coc}$ in the sense of Janelidze \cite{Jan}.
\begin{corollary}\label{cor2}
The categories $\mathsf{HXMod_{K, coc}}$ and $\mathsf{XMod}( \Hopf_{K, coc})$ are equivalent.
\end{corollary}
\begin{proof}
In any semi-abelian category the equivalence between internal groupoids and internal crossed modules was established in \cite{Jan}. The result then follows from Theorem \ref{semi-abelian} and Proposition \ref{equivalenceMajid}, since both the categories $\mathsf{HXMod_{K, coc}}$ and $\mathsf{XMod}( \Hopf_{K, coc})$ are equivalent to the category $\mathsf{Grpd}(\Hopf_{K, coc})$.
\end{proof}

\begin{corollary}
The categories $\mathsf{HXMod_{K, coc}}$ and $\mathsf{XMod}( \Hopf_{K, coc})$ are semi-abelian.
\end{corollary}
\begin{proof}
By Theorem \ref{semi-abelian} we know that $\Hopf_{K, coc}$ is semi-abelian, and this implies that the category $\mathsf{Grpd}(\Hopf_{K, coc})$ of internal groupoids in $\Hopf_{K, coc}$ is itself semi-abelian (see \cite{Gran}, and Lemma $4.1$ in \cite{BG2}). The result then follows from Proposition \ref{equivalenceMajid} and Corollary \ref{cor2}.\end{proof}

\subsection*{Acknowledgement}
The authors thank the referee for his/her useful remarks, which allowed us in particular to improve the proof of Proposition \ref{surjective}.

\bibliographystyle{plain}

\end{document}